\documentclass[12pt,a4paper]{article}
\usepackage[cp1251]{inputenc}
\usepackage{graphicx}
\usepackage{amssymb}
\usepackage{amsmath}
\usepackage{amsthm}
\usepackage{euscript}
\numberwithin{equation}{section}
\newtheorem{theorem}{Theorem}
\newtheorem{lemma}{Lemma}

\newtheorem{corollary}{Corollary}

\theoremstyle{definition}
\newtheorem{definition}{Definition}

\newcommand{\be}[0]{\begin{equation}}
\newcommand{\ee}[0]{\end{equation}}
\newcommand{\bez}[0]{\begin{equation*}}
\newcommand{\eez}[0]{\end{equation*}}
\newcommand{\bl}[0]{\begin{lemma}}
\newcommand{\el}[0]{\end{lemma}}

\newcommand{\ep}[0]{$\hspace{\fill} \square$}
\newcommand{\paragraf}[1]{\par
\bigskip{\centerline{\bf #1}}\medskip}

\newcommand{\abs}[1]{\begin{quotation} {\small
 \centerline{{\bf Abstract}}\smallskip
#1}
\end{quotation}}
\author{A.~Magazinov\footnote{Alfr\'{e}d R\'{e}nyi Institute of Mathematics}, I.~Shnurnikov\footnote{NRU HSE}}
\title{On the number of hyperbolic manifolds of complexity $n$}
\date{}
\begin{document}
\maketitle
\abs{
We consider hyperbolic manifolds with boundary, which admit an ideal triangulation with $n$ ideal triangles and one edge. We prove that the number of these manifolds is $\exp(n\ln(n)+O(n))$.
}

\paragraf{Introduction}

Frigerio, Martelli and Petronio in \cite{FRI} considered a class $M_n$ of 3--dimensional oriented manifolds with boundaries which admit an ideal triangulation with $n$ ideal triangulations and one edge. They proved that the complexity in Matveev sense of manifolds in $M_n$ equals $n$, that the manifolds could be supplied with hyperbolic metrics with geodesic boundaries, and that there are at least $O\left(\frac{6^n}n\right)$ manifolds in $M_n$. We present the asymptotic of the number $\ln |M_n|$:
$$
n\ln(n) +n\ln\left(\frac 2{3e}\right)+O(\ln(n)) \leq \ln |M_n| \leq n\ln(n) +n\ln(180)+O(\ln(n)).
$$

We use the correspondence (from \cite{FRI}) between number $|M_n|$ and the number of oriented special spines with $n$ vertices and one two--dimensional cell. We also use Bollobas' bound of the number of $r$--regular simple graphs with $n$ vertices.


\begin{definition}
A special spine is a finite connected two--dimensional cell complex, such that each vertex is incident to 4 edges (with multiplicities) and each edge is incident to three two--dimensional cells (with multiplicities). The regular neighborhood of the inner point of an edge is homeomorphic to a ``book with 3 pages'', the regular neighborhood of a vertex is homeomorphic to a cone on the edges of tetrahedron. A special spine is orientable, if it could be immersed into an oriented manifold.
\end{definition}

Let us call connected graphs without loops and multiple edges as simple graphs. Let us call graphs with all its vertices of fixed degree $
r$ as $r$--regular graphs. B.~Bollobas in \cite{Bol} estimated the number $|U_r(n)|$ of simple regular $r$--graphs with $n$ vertices (for even $nr$):

$$
\frac{e^{-\tfrac{r^2-1}{4}}(rn)!}{(\tfrac {rn}2)!2^{\tfrac{rn}2}(r!)^nn!}
$$
We will use the bound 
$$
\ln(U_4(n))=n \ln(n) +n\ln\left(\frac{2}{3e}\right)+O(\ln(n)).
$$


\paragraf{Main part.}

For $G \in A_n$ let $P(G)$ be a class of oriented special spines with singularity graph $G$ and with minimal number of cells (amoung all oriented special spines with singularity graph $G$).
For a spine $S \in P(G)$ let us choose two cells $(e_i,e_j)$ and count the number $v(e_i,e_j)$ of vertices, such that all incident to them edges belong to chosen cells.  Let $t(S)$ be the maximum of $v(e_i,e_j)$ for all pairs of cells $(e_i,e_j)$.

We shall consider a neighbourhood of $G$ in a spine $S$ with $r$ two--dimensional cells as a graph $G$ with $r$ glued cylinders, were one circle of cylinder is mapped into $G$ and the other is called the boundary line. 

\begin{definition}
An operation {\it rotation along an edge $e$} transforms a special oriented spine $P$ into a special oriented spine with the same singularity graph in the following way. Let the edge $e$ belong to cells $f_1, f_2$ and $f_3$. Let us consider instead of whole $P$ the neighbourhood of singularity graph in $P$. Then $f'_1$, $f'_2$ and $f'_3$ will denote cylinders, lying in $f_1, f_2$ and $f_3$ and in the neighbourhood of singularity graph. Let us cut the edge $e$ and the free boundaries of $f'_1$, $f'_2$ and $f'_3$. Let us glue again cutted free boundaries of $f'_1$, $f'_2$ and $f'_3$, but with a cyclic rotation on one side of the cut (i.e. $f'_1$ left with $f'_2$ right, $f'_2$ left with $f'_3$ right, $f'_3$ left with $f'_1$ right). Then we glue discs to obtained free boundaries of cylinders. Let us note that the number of two--dimensional cells may change.  
\end{definition}

\begin{lemma}
\label{lemma_2cells_per_edge}
Let $S \in P(G)$. Then every edge of $G$ belongs to at most 2 two--dimensional cells. 
\end{lemma}
\begin{proof} Let us suppose that an edge $e \in G$ belongs to 3 two--dimensional cells. We could cut the $e$ at the middlepoint and cut the boundary lines, which pass near $e$. If we fix an orientation of $e$, then we obtain a cyclic order of boundary lines near $e$. We rotate the parts of boundary lines clockwise on the one part of $e$ in such a way that a boundary line will glue with the next boundary line according to cyclic order. We glue rotated parts of boundary lines with unrotated ones. So we get a new oriented spine with the same singularity graph, and the number of cells decrease. It is a contradiction to $S \in P(G)$.
\end{proof}

\begin{lemma}
\label{parallel}
Let $S \in P(G)$ and an edge $e \in G$ belongs to 2 two--dimensional cells. Then if we choose an orientation on the boundary lines of each two--dimensional cell, then two boundary lines along $e$, which belong to one two--dimensional cell, will have parallel orientation. 
\end{lemma}
\begin{proof}
Let us  suppose the contrary. Analogously to the proof of lemma 1 we cut the edge $e$ in the middlepoint and cut the boundary lines. We rotate the parts of boundary lines clockwise on the one part of $e$ in such a way that a boundary line will glue with the next boundary line according to cyclic order. We glue rotated parts of boundary lines with unrotated ones. So we get a new oriented spine with the same singularity graph, and the number of cells decrease. It is a contradiction to $S \in P(G)$.
\end{proof}

\begin{lemma}
\label{lemma_2cellspervertex}
If $S \in P(G)$ then for every vertex $v \in G$ all incident to $v$ edges belong to at most 2 two--dimensional cells.
\end{lemma}
\proof
Suppose the contrary. Let $OA,OB,OC$ and $OD$ be 4 edges incident to a vertex $O \in G$.  

Without loss of generality we can assume that $OA$ belongs to 2 two--dimensional cells $e_1$ and $e_2$, and boundary line of $e_1$ passes through $AOC$ and boundary lines of $e_2$ pass through $AOB$ and $AOD$ correspondingly. There is at least one more cell $e_3$ with boundary passing through $O$.

Let us consider the following cases.

Case (i). $OB$ does not belong to a cell other than $e_1$ and $e_2$. Then $COD$ belongs to $e_3$. By lemma \ref{lemma_2cells_per_edge} $BOD$ belongs to $e_2$. By lemma \ref{lemma_2cells_per_edge} $BOC$ belongs to $e_1$. 
Considering the fragments $AOB$, $AOD$ and $BOD$ of the boundary of $e_2$ we get a contradiction, as we cannot choose an orientation on them to satisfy lemma \ref{parallel}.

Case (ii). The boundary of $e_3$ passes through $OB$. Then by lemma \ref{lemma_2cells_per_edge} $OB$ can belong only to the boundaries of
$e_2$ and $e_3$. Hence we have the following subcases:

(1) $e_3$ contains $COB$, $e_2$ contains $DOB$;

(2) $e_3$ contains $COB$ and $DOB$;

(3)  $e_3$ contains $DOB$, $e_2$ contains $COB$.

In the case (1) $COD$ cannot belong to $e_2$, so the edge $OD$ contradicts to lemma \ref{parallel}.

In the case (2) $COD$ belongs to $e_3$ and we have a contradiction to lemma \ref{parallel} in either $OC$ or $OD$.

In the case (3) $COD$ belongs to $e_2$.  Let us consider 4 segments $AOB, AOD, COB$ and $COD$ on the boundary of the cell $e_2$. The segment $AOB$ is a neighbour of two segments, at least one of which is not $COD$. So, without loss of generality we may assume that the boundary of $e_2$ passes through $AOB$ and then through $COB$ (not through $AOD$ or $COD$). Then we rotate along the edge $OB$ and glue $AOB$ with $BOD$ into cell $e'_3$. So the edge $AO$ belongs to three cells: $e_1, e'_2$ and $e'_3$, which contradicts to lemma \ref{lemma_2cells_per_edge}.
\ep 	

\begin{definition}
An operation of {\it gluing} a neighbourhood or singularity graph of the oriented spine $T$ with one two--dimensional cell into a neighbourhood or singularity graph of spine $S \in P(G)$ is defined in the following way.
Let us cut an edge $e\in G$ and consider 3 cutted boundary lines along $e$: $l_1, l_2$ and $l_3$. After we cut $e$ we get pairs $(l_1',l_1'')$, $(l_2',l_2'')$ and $(l_3',l_3'')$ of endpoints of $l_1,l_2$ and $l_3$.
Let us consider an oriented spine $T$ and cut an edge $f$ of singularity graph $T$.  Let $m_1, m_2$ and $m_3$ boundary lines along $f$. After we cut $f$ we get pairs of endpoints $(m_i',m_i'')$ for $i=1,2,3$. Suppose (it is a significant assumption) that if we travel from $(m_1',m_2',m_3')$ along cutted $m_1,m_2,m_3$ we will reach $(m_1'',m_2'',m_3'')$ (in another order, so the case $m_1'\rightarrow m_2'', m_3'\rightarrow m_1'', m_2'' \rightarrow m_3''$ is forbidden). We will call such edges $f$ {\it cutable}.

Then we could glue $l_i'$ to $m_i'$ and $l_i''$ to $m_i''$ and obtain a new spine $S'$. The number of two--dimensional cells of $S$ equals to the number of two--dimensional cells of $S'$. 
\end{definition}

\begin{lemma}
\label{Loop_orientation}
Let a graph $G$ be a singularity graph of an oriented spine. Let $A,B$ and $C$ be vertices of $G$ such that $G$ has edges $BA$ and $CA$ and a loop in a vertex $A$. Let $l_1, l_2$ and $l_3$ be boundary lines passing near edge $BA$. Let $m_1, m_2$ and $m_3$ be boundary lines passing near edge $CA$, so that $l_1, l_2, l_3$ passing through the loop turn to $m_1,m_2$ and $m_3$ correspondingly. Then the cyclic order of $l_1, l_2, l_3$ is different to the cyclic order of $m_1,m_2,m_3$. 
\end{lemma}
\proof
 It follows from the definition of oriented spine. 
\ep

\begin{lemma}
Let $P$ be an oriented special spine with singularity graph $G \in A_n$, with minimal number of two--dimensional cells (minimal among all oriented special spines with singularity graph $G$). Then $P$ has at most 2 two--dimensional cells.
\end{lemma}

\proof
Let us consider the contrary, that $P$ has at least 3 two--dimensional cells. Let us fix first 2 two--dimensional cells $f_1$ and $f_2$. Let us consider the set $U$ of vertices of $P$, such that all incident to them edges belong to fixed cells $f_1$ and $f_2$. If an edge $X_1X_2$ is incident to a vertex $x_1 \in U$ an to a vertex $X_2 \notin U$, then by lemma \ref{lemma_2cellspervertex} $X_1X_2$ belongs to only one of cells $f_1$ and $f_2$ (and doesn't belong to other cells). An edge $X_1X_2$ cannot belong to other cells because $X_1 \in U$. An edge $X_1X_2$ cannot belong to both cells $f_1$ and $f_2$ because an edge incident to $X_2$ belong to another cell and so vertex $X_2$ contradicts to lemma \ref{lemma_2cellspervertex}.

Let us consider a vertex $O \in U$, such that among incident to it edges $OA, OB, OC$ and $OD$ edge $OA$ belongs to $f_1$ and $f_2$, and $OC$ belongs to $f_1$ only. We will make several rotations along edges to obtain an oriented special spine with singularity graph $G$ and minimal number of two--dimensional cells, so that cells other from $f_1$ and $f_2$ will not change. Cells $f_1$ and $f_2$ will change to cells $f'_1$ and $f'_2$ with the same set $U'=U$. And the edge $OC$ will belong to both cells 
$f'_1$ and $f'_2$. The further proof is the following. Let $A_1,\dots, A_n$ be a sequence of vertices of $G$ such that $A_1, \dots, A_{n-1} \in U$, $A_n \notin U$ and an edge $A_1A_2$ belongs to both cells $f_1$ and $f_2$. Then we make rotations so that edge $A_{i-1}A_i$ belongs to two cells $f^{(i-2)}_1$ and $f^{(i-2)}_2$ with the same set $U^{(i-2)}=U$ for $i=2, \dots, n$. But $A_n \notin U$ --- contradiction.

So $OA$ belongs to $f_1$ and $f_2$, $OC$ belongs to $f_1$  only. We have the following cases.

\begin{enumerate}
\item $AOB \in f_2$, $AOC, AOD, BOD, BOC, COD \in f_1$ \\
\item $AOB, AOD \in f_2$, $AOC, BOD, BOC, COD \in f_1$ \\
\item $AOB, BOD \in f_2$, $AOC, AOD, BOC, COD \in f_1$ \\
\item $AOB, AOD, BOD \in f_2$, $AOC, BOC, COD \in f_1$ 
\end{enumerate}

In the case (2) we rotate along $OA$ and get the case (1).

Case (1). Let us consider the part $g_1$ of boundary of $f_1$, which passes from $BOC$ to $BOD$ (i.e. $BOC$ and $BOD$ divide the boundary of $f_1$ into 2 segments, one of which is considered). We also mean that if we go along the boundary of $f_1$ from $B$ to $C$ through the $BOC$ and further, then we will pass $g_1$ before we meet $BO$ once more. If both $AOC$ and $DOC$ belong to $g_1$, then we rotate along $OB$, gluing $AOB$ and $BOD$ in one cell. We will obtain the case (3) if $AOD \in g_1$ and the case (4) if $AOD \notin g_1$. Else we rotate along $OB$ and then $OC$ will belong to two cells $f'_1$ and $f'_2$. 

Case (3). Let us consider the part $g_1$ of boundary of $f_1$, which passes from $AOC$ to $AOD$ (i.e. $AOC$ and $AOD$ divide the boundary of $f_1$ into 2 segments, one of which is considered). If not both of $BOC$ and $DOC$ belong to $g_1$, then we rotate along $OA$ and get a spine such that $OC$ belongs to two cells. If both of $BOC$ and $DOC$ belong to $g_1$, then we rotate $OA$, gluing $AOD$ and $AOB$ in one cell, and obtain the case (4).

The case (4) is impossible by lemma \ref{parallel}.
\ep 

\begin{corollary}
For every graph $G \in A_n$ there exists a special oriented spine with singularity graph $G$ and with at most 2 two--dimensional cells.
\end{corollary}

\begin{theorem}
For $n \geq 8$ we have
$|M_n|\geq |A_{n-1}|$.
\end{theorem}
\begin{proof} 
Let us consider an arbitrary graph $G \in A_{n-1}$. Then there exists a special oriented spine $S$ on the graph $G$ with at most 2 two--dimensional cells. If $S$ has one two--dimensional cell then we find an edge of $G$ which is passed by the boundary of two--dimensional cell in different directions and glue it with a loop.
 If $S$ has two two--dimensional cells we find an edge $e$ which belongs to different cells  and we glue into $e$ a loop to obtain 
a special oriented spine with $n$ vertices and one two--dimensional cell by lemma \ref{Loop_orientation}.
\end{proof}

\begin{theorem}
$\ln|M_n| \leq \ln(|A_n|)+n (1+\ln(270))$.
\end{theorem}
\begin{proof}
Let $C_n$ be the set of connected homogeneous graphs with $n$ vertices of degree 4 (with loops and multiple edges). For each graph $G \in C_n$ there exist at most $18^n$ oriented special spines with singularity graph $G$. So $|M_n|\leq |C_n| \cdot 18^n$. By induction on $n$ one could prove that 
$$
\ln(\frac{|C_n|}{|A_n|}) \leq n(1+\ln(15)).
$$
\end{proof}

\end{document}